\documentclass[12pt,cd]{amsart}
\usepackage{latexsym}
\usepackage{amsmath,amsfonts,amssymb,amsthm,mathtools,amscd,bm,mathabx}
\usepackage{graphics}
\usepackage{csquotes}
\usepackage{blindtext}
\usepackage{xcolor}
 2

\usepackage{hyperref}
\newtheorem{theorem}{Theorem}[section]

\newtheorem{definition}[theorem]{Definition}

\newtheorem{remark}[theorem]{Remark}
\numberwithin{equation}{section}

\title{Decompositions of Scherk-Type Zero Mean Curvature Surfaces}

\author{Subham Paul}
\author{Priyank Vasu}
\author{Siddharth Panigrahi}
\author{Rahul Kumar Singh}

\address{Department of Mathematics, Indian Institute of Technology Patna, Bihta, Patna-801106, Bihar, India}

\thanks{Subham Paul: \href{mailto:subham_2021ma25@iitp.ac.in}{subham\_2021ma25@iitp.ac.in}}
\thanks{Priyank Vasu: \href{mailto:priyank_2121ma16@iitp.ac.in}{priyank\_2121ma16@iitp.ac.in}}
\thanks{Siddharth Panigrahi: \href{mailto:siddharth_2121ma12@iitp.ac.in}{siddharth\_2121ma12@iitp.ac.in}}
\thanks{Rahul Kumar Singh: \href{mailto:rahulks@iitp.ac.in}{rahulks@iitp.ac.in}}

\subjclass[2010]{53A10, 53B30, 11Z05}
\keywords{maximal surface; zero mean curvature surface; Scherk-type surface; Born-Infeld equation; Wick rotation; weakly untrapped surface;  $\Asterisk$-surface}
\begin{document}

\maketitle

\begin{abstract}

In this paper, by using a special Euler-Ramanujan identity and the idea of Wick rotation, we show that a one-parameter family of solutions to the zero mean curvature equation in Lorentz-Minkowski $3$-space $\mathbb E_1^3$, namely Scherk-type zero mean curvature surfaces, can be expressed as an infinite superposition of dilated helicoids. Further, we also obtain different finite decompositions for these surfaces. We end this paper with an application of these decompositions to formulate maximal codimension 2 surfaces into finite and infinite \enquote{sums} of weakly untrapped and $\Asterisk$-surfaces in Lorentz-Minkowski 4-space.

\end{abstract}

\section{Introduction}
 A surface $S$ in the Euclidean space $\mathbb{E}^3:= (\mathbb{R}^3,dx^2+dy^2+dz^2)$ is minimal if locally it can be expressed as a graph of some smooth function $ f(x,y)$ which satisfies the following partial differential equation:
	\begin{equation}\label{mse}
(1+f_{x}^2)f_{yy}-2f_{x}f_{y}f_{xy}+
(1+f_{y}^2)f_{xx}=0.
\end{equation}
The above equation is called the \textit{minimal surface equation} (this is also equivalent to the vanishing of the mean curvature function of $S$ in $\mathbb E^3$), the solution $z=f(x,y)$ of the equation \eqref{mse} is called the height function of the minimal surface $S$. 
Analogously, the height function \( g(x,y) \) of a zero mean curvature surface in the Lorentz-Minkowski 3-space $\mathbb{E}_1^3 := (\mathbb{R}^3, dx^2 + dy^2 - dz^2)$, when defined over a domain in the spacelike \( xy \)-plane, satisfies the following partial differential equation:
\begin{equation}\label{MaSE}
(1 - g_x^2)\, g_{yy} + 2 g_x g_y\, g_{xy} + (1 - g_y^2)\, g_{xx} = 0,
\end{equation}
We will call it the \textit{zero mean curvature equation} (ZMC equation). Moreover, if $g(x,y)$ satisfies $g_x^2+g_y^2<1$ on its domain, then the surface induced by $g$ is spacelike and the equation \eqref{MaSE} is also called the \textit{maximal surface equation}.   
Similarly, any solution \( h(y,z) \) to the equation
\begin{equation}\label{BIE}
(1 + h_y^2)\, h_{zz} - 2 h_y h_z\, h_{yz} + (h_z^2 - 1)\, h_{yy} = 0,
\end{equation}
yields the height function of a zero mean curvature surface in \( \mathbb{E}_1^3 \), defined over a domain in the timelike \( yz \)-plane. The equation \eqref{BIE} is also known as the \textit{Born-Infeld equation}, and the solutions to this equation are called \textit{Born-Infeld solitons} (see \cite{dey_bornInfeld}) 

It is not very hard to verify that the equations \eqref{mse}, \eqref{MaSE}, and \eqref{BIE} are interconnected through \textit{Wick rotations}, where a real variable (e.g., \( t \)) is replaced by an imaginary one (e.g., \( t \mapsto i t \)). The technique of changing a real variable to an imaginary variable is known as Wick rotation. In past, many authors have applied this technique to study the transformations of the solutions to these zero mean curvature equations. For more details on this, see \cite{akamine_wick, dey_bornInfeld, MALLORY20141669}.

It is a well-established fact now that there exists a connection between the height function of some special minimal surfaces and specific Euler-Ramanujan's identities;
the first such connection appears in a paper (see \cite{kamien2001decomposition}) by the physicist R. Kamien, where he uses a particular Euler-Ramanujan's identity to show that the Scherk's surface (a minimal surface) can be decomposed into an infinite superposition of  ``dilated helicoids".
Another connection between Euler-Ramanujan identities and zero mean curvature surfaces also appears in the work of R.~Dey and her collaborators (see \cite{dey_bornInfeld, dey_ramanujan}) where the authors use the Weierstrass–Enneper representation of certain classes of minimal and maximal surfaces in conjunction with specific Euler–Ramanujan identities to derive new complex identities. R.~Kamien, in his paper \cite{kamien2001decomposition}, suggested a possible decomposition of the solutions to equation~\eqref{BIE} which are related to the solutions of \eqref{mse} through Wick rotation. However, a key difficulty arises in this context: after Wick rotation, the solutions to~\eqref{BIE} typically become complex-valued, making them unsuitable as real-valued height functions. In a subsequent study, the fourth author with S.~Akamine, in \cite{akamine_wick}, investigated the conditions under which real (purely imaginary) solutions of all three ZMC equations mentioned above persist after Wick rotations. They showed that these conditions can be characterized by certain symmetries of the solutions.

 In this paper, we perform a Wick rotation of the Scherk's minimal surfaces presented in \cite{kamien2001decomposition}, and obtain real solutions to both the maximal surface equation and the Born-Infeld equation. We then decompose these surfaces using the Euler-Ramanujan identity as in \cite{berndt2012ramanujan}, resulting in infinite
decompositions into dilated helicoids and dilated hyperbolic helicoids, respectively. Furthermore, we establish finite decompositions of the aforementioned surfaces into dilated versions of the same surfaces. These results provide partial answers to the questions posed by R.~Kamien in \cite{kamien2001decomposition}. Lastly as an application of these decompositions, we express spacelike maximal surfaces in $\mathbb{E}_1^4$ as finite and infinite \enquote{sums} of weakly untrapped and $\Asterisk$-surfaces. The question of obtaining similar decompositions of general solutions to the ZMC equations $\eqref{mse}$, $\eqref{MaSE}$, and $\eqref{BIE}$ remains to be explored.

This paper is organized as follows. Section~\ref{Preliminaries} provides the necessary background, including fundamental definitions and results. In Section~\ref{Infinite decomposition}, we establish the infinite decomposition of Scherk's type ZMC surfaces in $\mathbb E_1^3$. Section~\ref{Finite decomposition} is devoted to the corresponding finite decomposition results. Finally, in Section~\ref{appl}, we discuss the aforementioned application of these decompositions.

       \section{Preliminaries}\label{Preliminaries}

Let \( \mathbb{E}^3  \) denote $\mathbb{R}^3$, endowed with the metric \( \mathrm{d}x^2 + \mathrm{d}y^2 + \mathrm{d}z^2 \) and the Lorentz-Minkowski 3-space, denoted by \( \mathbb{E}^3_1  \), is \( \mathbb{R}^3 \) equipped with the Lorentzian metric \( \mathrm{d}x^2 + \mathrm{d}y^2 - \mathrm{d}z^2 \). In both spaces, \( (x, y, z) \) are the canonical coordinates of \( \mathbb{R}^3 \). Throughout this paper, all real-valued functions are assumed to be real analytic, and \( 'i' \) denotes the imaginary unit, that is, \( i = \sqrt{-1} \).

 An immersion $X: \Omega \rightarrow \mathbb{E}^3_1$ of a domain $\Omega \subset \mathbb{R}^2$ into $\mathbb{E}^3_1$ is called \textit{spacelike (resp. timelike, lightlike)} at a point $p \in \Omega$ if its first fundamental form $\mathrm{I}=X^*\langle,\rangle$ is \textit{Riemannian (resp. Lorentzian, degenerate)} at $p$. A spacelike zero mean curvature surface is known as a \emph{maximal surface}, whereas a timelike one is referred to as a \emph{timelike minimal surface}.

A surface in \( \mathbb{E}^3 \) defined as the graph of a function \( f: \Omega \subset \mathbb{R}^2 \to \mathbb{R} \) on the \( xy \)-plane takes the form
\[
F(x, y) = (x, y, f(x, y)),
\]
and is minimal if \( f \) satisfies the minimal surface equation~\eqref{mse}.

In \( \mathbb{E}^3_1 \), a graph on the spacelike \( xy \)-plane is expressed as
\[
G(x, y) = (x, y, g(x, y)),
\]
and is a zero mean curvature surface if \( g \) satisfies the zero mean curvature equation~\eqref{MaSE}. In addition, the surface becomes spacelike if \( g_x^2 + g_y^2 < 1 \), making it a maximal surface, and timelike if \( g_x^2 + g_y^2 > 1 \), which makes it a timelike minimal surface and it degenerates when $g_x^2 + g_y^2=1$.
Similarly, a graph on the timelike \( yz \)-plane is written as
\[
H(y, z) = (h(y, z), y, z),
\]
and is a zero mean curvature surface if  \( h \) satisfies the Born-Infeld equation~\eqref{BIE}. It defines a maximal surface when \( 1 + h_y^2 - h_z^2 < 0 \), and a timelike minimal surface when \( 1 + h_y^2 - h_z^2 > 0 \).

 The relation between these graphs has been studied by the fourth author with his collaborators in \cite{akamine_wick,dey_bornInfeld}, where it was shown that they are connected through Wick rotations. For instance, consider the following example:

 \begin{equation}\label{Scherk minimal surface}
 z = f(x, y) = \log \left( \frac{\cos x}{\cos y} \right),
 \end{equation}
 is a well-known example of a solution to the minimal surface equation \eqref{mse}, called the \textit{classical Scherk's minimal surface}. Now, if we replace $x$ and $y$ with $ix$ and $iy$ in \eqref{Scherk minimal surface}, we get the function

\begin{equation}
z = g(x, y) = f(ix,iy)= \log \left( \frac{\cosh x}{\cosh y}\right),\end{equation}
 which solves the maximal surface equation \eqref{MaSE} in $\mathbb{E}^3_1$. On a similar note, if we replace only $y$ with $iy$ in \eqref{Scherk minimal surface} and redefine the coordinates as $(x, y, z) \mapsto ({y}, {z}, {x})$, we obtain the following solution to the Born–Infeld equation \eqref{BIE} in $\mathbb{E}^3_1$ as follows:  

\begin{equation}
{x} = h({y}, {z}) =f(y,iz)= \log \left( \frac{\cos {y}}{\cosh {z}} \right).
\end{equation}

  These are not merely ad hoc examples; one can verify that substituting \( x \mapsto i x \) and \( y \mapsto i y \) in a solution of the minimal surface equation~\eqref{mse} produces a solution of the maximal surface equation~\eqref{MaSE}, and vice versa.
 Similarly, if $f(x,y)$ is a solution to the minimal surface equation in $\mathbb E^3$, then $f(y,iz)$ is a solution to the Born-Infeld equation $\eqref{BIE}$ in $\mathbb E_1^3$. However, the resulting solutions are complex-valued in general. For example, consider the height function of \textit{helicoid}:
\begin{equation}\label{minimal helicoid}
    \xi(x, y) = \tan^{-1}\left(\frac{y}{x}\right).
\end{equation}
This function \( \xi(x, y) \) satisfies both the minimal surface equation~\eqref{mse} and the maximal surface equation~\eqref{MaSE}.

Now, applying Wick rotation, we obtain
\begin{equation}\label{complex wick helicoid}
    \varphi(y, z) = \xi(y,iz)= i\tanh^{-1}\left(\frac{z}{y}\right),
\end{equation}
which defines a complex-valued solution to the Born-Infeld equation~\eqref{BIE}. We refer to the surface described by~\eqref{complex wick helicoid} as \emph{Wick helicoid}.
 The fourth author, together with S. Akamine in \cite{akamine_wick} gives the criteria for the existence of real and imaginary solutions after Wick
rotations as follows:

\subsection{Transformation between solutions of \eqref{mse} and \eqref{MaSE}
}

\begin{theorem}\label{minimal to maximal}\cite{akamine_wick}
    Let $f(x, y)$ be a solution to \eqref{mse} in $\mathbb{E}^3$. If $f$ is even (resp. odd) with respect to the $x$-and $y$-variables, then its Wick rotation $g(x, y)=f(i x, i y)$ (resp. $g(x, y)=-f(i x, i y))$ is a real solution to \eqref{MaSE}   in $\mathbb{E}^3_1$ which is spacelike at least near the origin $o=(0,0)$ and conversely.
\end{theorem}

\subsection{Transformations between solutions of \eqref{mse} and \eqref{BIE}}

\begin{theorem}\footnote{Common variable names are used in the definitions of related functions for notational convenience, even though their interpretation vary depending on the context.}\label{even minimal to born infi}\cite{akamine_wick}
    Let $f(x, y)$ be a solution to \eqref{mse} in $\mathbb{E}^3$, which is even with respect to the $y$-variable, then $h(y,z)=$ $f(y, iz)$ is a real solution to \eqref{BIE} in $\mathbb{E}^3_1$. Conversely, for a solution $h(y,z)$ to \eqref{BIE} in $\mathbb{E}^3_1$, which is even with respect to the $z$-variable, $f(x,y)=h(x, i y)$ is a real solution to \eqref{mse} in $\mathbb{E}^3$.
\end{theorem}

\begin{theorem}\label{odd minimal to born infi}\cite{akamine_wick}
    Let $f(x, y)$ be a solution to \eqref{mse} in $\mathbb{E}^3$, which is odd with respect to the $y$-variable, then $h(y,z)=-i f({z},  i {y})$ is a real solution to \eqref{BIE} in $\mathbb{E}^3_1$. Conversely, for a solution $h({y}, {z})$ to \eqref{BIE} in $\mathbb{E}^3_1$ which is odd with respect to the ${y}$-variable, $f(x,y)=-i h(i y,x)$ is a real solution to \eqref{mse} in $\mathbb{E}^3$.
\end{theorem}
Before proceeding further, we introduce the concept of a dilated surface, which has been frequently used in \cite{kamien2001decomposition,dey_finite} but has not been formally defined. Our aim here is to provide a precise definition.

\begin{definition}
    Let \( H(x, y) = \left(x, y, f(x, y)\right) \) be a parametrization of a surface \( S \) defined on a domain in \( \mathbb{R}^2 \). For constants \( k, m_1, m_2, n_1, n_2 \in \mathbb{R} \), we define the surface
    \[
        \tilde{H}(\tilde{x}, \tilde{y}) = \left(\tilde{x}, \tilde{y}, k \cdot f(m_1 \tilde{x} + n_1, m_2 \tilde{y} + n_2)\right)
    \]
   in an appropriate domain to be a dilated version of \( S \), or simply dilated $S$.
\end{definition}
The above definition can be analogously extended by allowing the constants in the expression of $\tilde{H}$ to take complex values instead of real ones. The dilation, although it preserves the topology of the surface locally but changes the geometry. For simplicity of notation, throughout this paper, we adopt the following convention for bilateral infinite sums: for a (real or complex) sequence $\{x_k\}_{k=-\infty}^{\infty}$, we define
\[
\sum_{k=-\infty}^{\infty} x_k := \sum_{k=1}^{\infty} x_{-k} + \sum_{k=0}^{\infty} x_k,
\]
provided that both sums on the right-hand side (RHS) converge.

\subsection{Euler-Ramanujan identity}\cite{berndt2012ramanujan}
For all $a,b \in \mathbb{R}$ we have the following identity:
\begin{equation}\label{ERidentity}
\tan ^{-1}[\tanh a \cot b]=\sum_{k=-\infty}^{\infty} \tan ^{-1}\left(\frac{a}{b+k \pi}\right). 
 \end{equation}

The above identity \eqref{ERidentity} can also be extended in a suitable domain $(a,b)\in\mathbb C^2$, for which $\tan^{-1}$ is a well-defined analytic function.


\section{Infinite decomposition of Scherk's type surfaces}\label{Infinite decomposition}

We start this section with the following solutions to the minimal surface equation \eqref{mse} as given in \cite{kamien2001decomposition}:

\begin{equation}\label{familyscherk}
    \phi(x,y;\theta)=-\sec{\frac{\theta}{2}}	\tan^{-1}\left(\dfrac{\tanh (\frac{x\sin \theta}{2})}{\tan(y\sin\frac{\theta}{2})}\right).
    \end{equation}
This gives a family of \textit{Scherk's minimal surfaces} for $\theta \in (0,\frac{\pi}{2})$. Note that in the limit as $\theta \rightarrow 0$, the height functions of this family of Scherk's minimal surfaces converge to the height function of the helicoid. Indeed,
\begin{equation}\label{limit1}
    \lim_{\theta \rightarrow 0} \phi(x,y;\theta)  = - \lim_{\theta \rightarrow 0}  \sec\left(\frac{\theta}{2} \right) \tan^{-1} \left( \frac{\theta x/2}{\theta y/2}\right) 
    = - \tan^{-1} \left( \frac{ x}{y}\right) 
    = -\epsilon_{x,y} \, \frac{\pi}{2} + \tan^{-1}\left(\frac{y}{x} \right),
 \end{equation}   
 where $\epsilon_{x,y}$ is +1 if $x/y >0$ and $\epsilon_{x,y}$ is -1 if $x/y <0$.\\
It is easy to check that $\phi(x,y;\theta)$ is odd in both the variables, and therefore, from Theorem \ref{minimal to maximal}, we get a family of \textit{Scherk's maximal surfaces} given by  

\begin{equation}\label{familyscherkmaximal}
    \psi(x,y;\theta):=-\phi(ix,iy;\theta)=\sec{\frac{\theta}{2}}	\tan^{-1}\left(\dfrac{\tan (\frac{x\sin \theta}{2})}{\tanh(y\sin\frac{\theta}{2})}\right).
    \end{equation}

 In other words, $\psi(x,y;\theta)$ gives a one-parameter family of maximal surfaces around some neighbourhood of the origin. Similarly, from Theorem \ref{odd minimal to born infi}, we get a one-parameter family of solutions of the Born-Infeld equation, given by  

\begin{equation}\label{familystimelike}
    \chi(y,z;\theta):=-i\phi(z,iy;\theta)=\sec{\frac{\theta}{2}}	\tanh^{-1}\left(\dfrac{\tanh (\frac{z\sin \theta}{2})}{\tanh(y\sin\frac{\theta}{2})}\right).
    \end{equation}

We will call this family \textit{Scherk's type Born-Infeld solitons}. It is important to note that the solutions obtained are all real-valued. For these two families, we have the limits as follows:
\begin{align}
 \label{limit2}  \lim_{\theta \rightarrow 0} \psi(x,y;\theta)  &= \tan^{-1} \left( \frac{ x}{y}\right) , \\
   \lim_{\theta \rightarrow 0} \chi(y,z;\theta)  &=  \tanh^{-1} \left( \frac{z}{y}\right) \label{limit3}.
\end{align}

The limiting surface in equation \eqref{limit3}, is known as the \textit{hyperbolic helicoid}. Using the Euler-Ramanujan identity \eqref{ERidentity}, we now present an infinite decomposition result for the one-parameter family of Scherk's maximal surfaces, expressing them as a superposition of dilated helicoids:

\begin{theorem}
    The height functions of the one-parameter family of Scherk's maximal surfaces can be decomposed as an infinite superposition of dilated helicoids as follows:
    \begin{equation}
        \label{decompositionmaximal}
       \psi(x,y;\theta)=\sec{\frac{\theta}{2}}	\left[\frac{\pi}{2}-\sum_{n =-\infty}^{\infty}\tan^{-1}\left(\dfrac{y}{x\cos\frac{\theta}{2}-ns}\right)\right],
    \end{equation}
    where $\sin \frac{\theta}{2}=\frac{\pi}{s}$ and $\dfrac{\tan (\frac{x\sin \theta}{2})}{\tanh(y\sin\frac{\theta}{2})}>0$.
\end{theorem}

\begin{proof}
First, we see 
\begin{align*}
    \tan^{-1}\left(\dfrac{\tan (\frac{x\sin \theta}{2})}{\tanh(y\sin\frac{\theta}{2})}\right) 
    &= \frac{\pi}{2}-\tan^{-1}\left(\dfrac{\tanh(y\sin\frac{\theta}{2})}{\tan (\frac{x\sin \theta}{2})}\right)\\
    &=\frac{\pi}{2}-\sum_{n=-\infty}^{\infty}\tan^{-1}\left(\dfrac{y\sin\frac{\theta}{2}}{\frac{x}{2}\sin \theta+n\pi}\right), \quad \text{using \eqref{ERidentity}}\\
    &=\frac{\pi}{2}-\sum_{n=-\infty}^{\infty}\tan^{-1}\left(\dfrac{y}{x\cos \frac{\theta}{2}+ns}\right), 
   \end{align*}
where $\sin \frac{\theta}{2}=\frac{\pi}{s}$.\\

By changing the summation index \( n \to -n \) and observing that the summation ranges over all integers \( n \in \mathbb{Z} \), we obtain:

\begin{align*}
   \psi(x,y;\theta)& = \sec{\frac{\theta}{2}}	\tan^{-1}\left(\dfrac{\tan (\frac{x\sin \theta}{2})}{\tanh(y\sin\frac{\theta}{2})}\right)\\
   &=\sec{\frac{\theta}{2}}	\left[\frac{\pi}{2}-\sum_{n =-\infty}^{\infty}\tan^{-1}\left(\dfrac{y}{x\cos\frac{\theta}{2}-ns}\right)\right].
\end{align*}

   Hence, the claim. 
\end{proof}

The next theorem shows that the height function in \eqref{familystimelike} can also be decomposed into a superposition of dilated hyperbolic helicoids:

\begin{theorem}
  The height functions of Scherk's type Born-Infeld solitons can be decomposed as an infinite superposition of dilated hyperbolic helicoids as follows:
    \begin{equation}
        \label{decompositiontimelike}
       \chi(y,z;\theta)=\sec{\frac{\theta}{2}}\sum_{n =-\infty}^{\infty}\tanh^{-1}\left(\dfrac{z\cos\frac{\theta}{2}}{y-ns}\right),
    \end{equation}
    where $\sin \frac{\theta}{2}=\frac{i\pi}{s}$.  
\end{theorem}

\begin{proof}
Firstly, we have
\begin{align*}
    \tan^{-1}\left(\dfrac{\tanh (\frac{z\sin \theta}{2})}{\tan(iy\sin\frac{\theta}{2})}\right) 
    &=\sum_{n=-\infty}^{\infty}\tan^{-1}\left(\dfrac{\frac{z\sin\theta}{2}}{iy\sin \frac{\theta}{2}+n\pi}\right), \quad \text{using \eqref{ERidentity}}\\
    &=-i\sum_{n=-\infty}^{\infty}\tanh^{-1}\left(\dfrac{\frac{z\sin\theta}{2}}{y\sin \frac{\theta}{2}-in\pi}\right).
\end{align*}
Therefore,
\begin{align*}
    \chi(y,z;\theta) &= -i\phi(z,iy;\theta) \\
    &= -i\left(-\sec{\frac{\theta}{2}}\right) \tan^{-1}\left(\dfrac{\tanh (\frac{z\sin \theta}{2})}{\tan(iy\sin\frac{\theta}{2})}\right)\\
    &= \sec \frac{\theta}{2} \sum_{n=-\infty}^{\infty} \tanh^{-1}\left(\dfrac{\frac{z\sin\theta}{2}}{y\sin \frac{\theta}{2}-in\pi}\right), \quad \text{using \eqref{ERidentity}}\\
    &= \sec \frac{\theta}{2} \sum_{n=-\infty}^{\infty} \tanh^{-1}\left(\dfrac{z\cos\frac{\theta}{2}}{y-in\pi\csc \frac{\theta}{2}}\right).
\end{align*}
 Hence, the claim. 
\end{proof}
\begin{remark}
R.~Kamien, in \cite{kamien2001decomposition}, observed that Scherk’s minimal surfaces can be interpreted as an infinite superposition of ``topological defects'', which correspond to the limiting helicoids described in equation~\eqref{limit1}. Analogously, in the preceding two theorems, we find that the families represented by $\psi$ and $\chi$ can also be understood as infinite superpositions of their respective limits, given by equations~\eqref{limit2} and~\eqref{limit3}.
     In equation \eqref{decompositiontimelike}, the left-hand side (LHS) is real-valued, whereas the right-hand side (RHS) is expressed in terms of complex-valued functions. Therefore, 
    \[
    \sum_{n = -\infty}^{\infty}
    \operatorname{Im}\left[ \tanh^{-1}\left( \frac{z \cos\frac{\theta}{2}}{y - ns} \right) \right] = 0.
    \]
    Consequently, although complex functions appear in the intermediate steps, the real height function $\chi(y, z; \theta)$ ultimately decomposes into a sum of real-valued height functions. By taking the real part of both sides of \eqref{decompositiontimelike}, we obtain the following identity:
\begin{align*}
\chi(y, z; \theta) 
&= \sec\left( \frac{\theta}{2} \right) \sum_{n = -\infty}^{\infty} 
\operatorname{Re}\left[ \tanh^{-1}\left( \frac{z \cos\frac{\theta}{2}}{y - ns} \right) \right] \\
&= \frac{1}{2} \sec\left( \frac{\theta}{2} \right) 
\sum_{n = -\infty}^{\infty} \log \left| \frac{ y - ns + z \cos\frac{\theta}{2} }{ y - ns - z \cos\frac{\theta}{2} } \right|.
\end{align*}
\end{remark}

\section{Finite decomposition of Scherk's type surfaces}\label{Finite decomposition}


R.~Kamien, in \cite{kamien2001decomposition}, expressed a family of Scherk’s minimal surfaces as a finite sum of dilated copies of itself. Inspired by this idea, R.~Dey et al.\ \cite{dey_finite} investigated a similar decomposition of classical Scherk's minimal surface~\eqref{Scherk minimal surface}. Their approach employed a different Euler–Ramanujan identity (not the one used in our work, cf.\ \eqref{ERidentity}), leading to a finite decomposition of the height function in terms of ``translated helicoids’’ and ``scaled and translated Scherk’s first surfaces’’. Our results are more closely aligned with Kamien’s original work and give decomposition of Scherk’s minimal surfaces and the surfaces obtained by it's Wick rotations (i.e., $\phi,\psi \text{ and } \chi$). Moreover, Dey’s work gives the theory of solutions of zero mean curvature equations \eqref{mse}, \eqref{MaSE}, and \eqref{BIE} that split into a finite sum of dilated versions of the same. But, our result (Theorem \ref{theorem finite decomposition}) obtained in this section highlights that  \( \phi \),  \( \psi \), and  \( \chi \) can be decomposed into finite sums of dilated versions of one another. Related work appears in \cite{paul2024superposition}, where the authors provide sufficient conditions for finite sums of parametrized maximal surfaces via their Weierstrass-Enneper representation. \par
Now, rewriting the RHS expression of \eqref{ERidentity}, we get the following finite sum of infinite series: for any positive integer $n$,
 \begin{align}\label{NEWERidentity1}
\sum_{k=-\infty}^{\infty} \tan ^{-1}\left(\frac{a}{b+k \pi}\right)&= \sum_{m=0}^{n-1}\sum_{j=-\infty}^{\infty}\tan^{-1}\left(\frac{a}{b+(nj+m)\pi}\right)\\\label{NEWERidentity2}
&=\sum_{m=0}^{n-1}\tan^{-1}\left(\tanh\left(\frac{a}{n}\right)\cot\left(\frac{b+m\pi}{n}\right)\right).
 \end{align}

In the following theorem, we make use of the identity \( \tan^{-1}x + \tan^{-1}\left(\frac{1}{x}\right) = \frac{\pi}{2} \) for \( x > 0 \). Accordingly, we restrict the domain of the functions \( \phi, \psi, \chi \) to ensure the validity of this identity, without loss of generality. In certain cases of the following theorem, we can see dilated forms of \( \phi \) and \( \psi \) after interchanging the variables \( x \) and \( y \), which preserves their nature as minimal and maximal surfaces, respectively.

\begin{theorem}\label{theorem finite decomposition}
    We have the following finite decompositions:
    \begin{enumerate}
        \item $\psi(\sec \beta.x,y;2\beta)=(1-n)\frac{\pi}{2}\sec\beta+\frac{\cos \tilde{\beta}}{\cos \beta}.\sum_{m=0}^{n-1}\psi(\frac{2x}{n}+\frac{2m\pi\csc2\tilde{\beta}}{n},2y\cos2\tilde{\beta};2\tilde{\beta})$, \quad{where $\sin2\tilde{\beta}=\frac{\sin \beta}{n}$}.
        \item $\psi(\sec \beta\sec 2\beta.x,y;4\beta)=\frac{\pi}{2}\sec 2\beta+\frac{\cos \beta}{\cos 2\beta}.\sum_{m=0}^{n-1}\phi(\frac{2y}{n},\frac{2x}{n}+\frac{m\pi}{n}\csc \beta;2\beta)$.
        \item $\phi(\sec \beta\sec 2\beta.x,y;4\beta)=-\frac{n\pi}{2}\sec 2\beta+\frac{\cos \beta}{\cos 2\beta}.\sum_{m=0}^{n-1}\psi(\frac{2y}{n}+\frac{2m\pi}{n}\csc 2\beta,\frac{2x}{n};2\beta)$.
        \item $\chi(y,z;2\beta)=\sum_{m=0}^{n-1}\chi\left(\frac{y}{n}-i\frac{m\pi}{n}\csc\beta,\frac{z}{n};2\beta\right)$.
    \end{enumerate}
\end{theorem}
\begin{proof}
    \begin{enumerate}
    
   \item From equation $\eqref{familyscherkmaximal}$, we have
   \begin{align*} 
       \psi(x,y;\theta)&=\sec{\frac{\theta}{2}}	\tan^{-1}\left(\dfrac{\tan (\frac{x\sin \theta}{2})}{\tanh(y\sin\frac{\theta}{2})}\right)\\
        &=\sec{\frac{\theta}{2}}\left[\frac{\pi}{2}-\sum_{n=-\infty}^{\infty}\tan^{-1}\left(\dfrac{y\sin\frac{\theta}{2}}{\frac{x}{2}\sin \theta+n\pi}\right)\right], \quad \text{using \eqref{ERidentity}}.
          \end{align*}
        Substituting $\theta=2\beta$ and using the identity \eqref{NEWERidentity2} in the above expression we get,
        \begin{align*}
            \cos\beta .\psi(x,y;2\beta)&= \frac{\pi}{2}-\sum_{m=0}^{n-1}\tan^{-1}\left(\dfrac{\tanh(\frac{y\sin\beta}{n})}{\tan (\frac{x\sin \beta\cos\beta+m\pi}{n})}\right)\\
            &=\frac{\pi}{2}-\sum_{m=0}^{n-1}\tan^{-1}\left(\dfrac{\tanh(y\sin2\tilde{\beta})}{\tan (\frac{x\sin 2\tilde{\beta}\cos\beta+m\pi}{n})}\right), \text{where } \sin 2\tilde{\beta} = \frac{\sin \beta}{n}\\
            &= \frac{\pi}{2} - \sum_{m=0}^{n-1} \left[ \frac{\pi}{2} - \tan^{-1} \left( \frac{\tan ( \frac{x \sin 2\tilde{\beta} \cos\beta + m\pi}{n} }{\tanh(y \sin 2\tilde{\beta})} \right) \right]\\
            &=(1-n)\frac{\pi}{2}+\cos\tilde{\beta}\sum_{m=0}^{n-1}\psi\left(\frac{2x\cos\beta}{n}+\frac{2m\pi\csc2\tilde{\beta}}{n},2y\cos\tilde{\beta};2\tilde{\beta}\right).
 \end{align*}
 Thus, we have established the first part of the theorem.

   \vspace{1cm}
    \item Again, starting from equation $\eqref{familyscherkmaximal}$, we have
    \begin{align*}
       \psi(x,y;\theta)&=\sec{\frac{\theta}{2}}	\tan^{-1}\left(\dfrac{\tan (\frac{x\sin \theta}{2})}{\tanh(y\sin\frac{\theta}{2})}\right)\\
        &=\sec{\frac{\theta}{2}}\left[\frac{\pi}{2}-\sum_{n=-\infty}^{\infty}\tan^{-1}\left(\dfrac{y\sin\frac{\theta}{2}}{\frac{x}{2}\sin \theta+n\pi}\right)\right], \quad \text{using \eqref{ERidentity}}.
          \end{align*}
    Now
    \begin{align*}
        \sum_{k=-\infty}^{\infty}\tan^{-1}\left(\dfrac{y\sin\frac{\theta}{2}}{\frac{x}{2}\sin \theta+k\pi}\right)&=\sum_{m=0}^{n-1}\tan^{-1}\left(\dfrac{\tanh(\frac{y}{n}\sin\frac{\theta}{2})}{\tan\left(\frac{x}{2n}\sin \theta+\frac{m\pi}{n}\right)}\right), \quad \text{using \eqref{NEWERidentity2}}\\
        &=-\cos \frac{\theta}{4}\sum_{m=0}^{n-1}\phi\left(\frac{2y}{n},\frac{2x}{n}\cos\frac{\theta}{4}\cos\frac{\theta}{2}+\frac{m\pi}{n}\csc\frac{\theta}{4};\frac{\theta}{2}\right).
    \end{align*}
    Therefore 
    \begin{align*}
        \psi(x,y;\theta)&=\sec\frac{\theta}{2}\left[\frac{\pi}{2}+\cos \frac{\theta}{4}\sum_{m=0}^{n-1}\phi\left(\frac{2y}{n},\frac{2x}{n}\cos\frac{\theta}{4}\cos\frac{\theta}{2}+\frac{m\pi}{n}\csc\frac{\theta}{4};\frac{\theta}{2}\right)\right].
    \end{align*}
Now, substituting $x$ by $\sec\beta\sec2\beta.x$ and $\theta$ by $4\beta$ in the above expression, we establish the second part of the theorem.
    
 \vspace{1cm}
    \item From equation $\eqref{familyscherk}$, we have
    \begin{align*}
         \phi(x,y;\theta)&=-\sec{\frac{\theta}{2}}	\tan^{-1}\left(\dfrac{\tanh (\frac{x\sin \theta}{2})}{\tan(y\sin\frac{\theta}{2})}\right).
    \end{align*}
    Now
    \begin{align*}
        \tan^{-1}\left(\dfrac{\tanh (\frac{x\sin \theta}{2})}{\tan(y\sin\frac{\theta}{2})}\right)&=\sum_{k=-\infty}^{\infty}\tan^{-1}\left(\dfrac{\frac{x}{2}\sin\theta}{y\sin\frac{\theta}{2}+k\pi}\right), \quad \text{using \eqref{ERidentity}}\\
        &=\sum_{m=0}^{n-1}\tan^{-1}\left(\dfrac{\tanh\left(\frac{x}{2n}\sin\theta\right)}{\tan\left(\frac{y}{n}\sin\frac{\theta}{2}+\frac{m\pi}{n}\right)}\right), \quad \text{using \eqref{NEWERidentity2}}\\
        &=\sum_{m=0}^{n-1}\left[\frac{\pi}{2}-\tan^{-1}\left(\dfrac{\tan\left(\frac{y}{n}\sin\frac{\theta}{2}+\frac{m\pi}{n}\right)}{\tanh\left(\frac{x}{2n}\sin\theta\right)}\right)\right]\\
        &=\frac{n\pi}{2}-\cos\frac{\theta}{4}\sum_{m=0}^{n-1}\psi\left(\frac{2y}{n}+\frac{2m\pi\csc\frac{\theta}{2}}{n}, \frac{2x}{n}\cos\frac{\theta}{2}\cos\frac{\theta}{4};\frac{\theta}{2}\right).
    \end{align*}
    
    Therefore
    \begin{equation*}
        \phi(x,y;\theta)=-\sec\frac{\theta}{2}\left[\frac{n\pi}{2}-\cos\frac{\theta}{4}\sum_{m=0}^{n-1}\psi\left(\frac{2y}{n}+\frac{2m\pi\csc\frac{\theta}{2}}{n}, \frac{2x}{n}\cos\frac{\theta}{2}\cos\frac{\theta}{4};\frac{\theta}{2}\right)\right].
    \end{equation*}
    Now, substituting $x$ by $\sec\beta\sec2\beta.x$ and $\theta$ by $4\beta$ in the above expression, we get the third part of the theorem.
    
     \vspace{1cm}
    \item From equation $\eqref{familystimelike}$, we have
    \begin{align*}
    \chi(y,z;\theta)=-i\phi(z,iy;\theta)&=i\sec{\frac{\theta}{2}}	\tan^{-1}\left(\dfrac{\tanh (\frac{z\sin \theta}{2})}{\tan(iy\sin\frac{\theta}{2})}\right)\\
    &=i\sec{\frac{\theta}{2}}	\sum_{m=0}^{n-1}\tan^{-1}\left(\dfrac{\tanh (\frac{z\sin \theta}{2n})}{\tan\left(\frac{iy\sin\frac{\theta}{2}+m\pi}{n}\right)}\right),\quad \text{using \eqref{ERidentity}, \eqref{NEWERidentity2}}\\
    &=\sum_{m=0}^{n-1}\chi\left(\frac{y}{n}-i\frac{m\pi}{n}\csc\frac{\theta}{2},\frac{z}{n};\theta\right).
    \end{align*}
    So, by substituting $\theta$ by $2\beta$ in the above expression, we are done.
\end{enumerate}
\end{proof}

It can also be observed that the height function \( \chi \) admits a similar finite decomposition in terms of dilated copies of \( \phi \) and \( \psi \). 

\section{Decomposition of Maximal Surfaces in $\mathbb E_1^4$ into Untrapped and $\Asterisk$-Dilated Maximal Surfaces}\label{appl}

    In this section, we will present some decompositions of the maximal surfaces by dilated maximal surfaces (which happen to be weakly untrapped and $\Asterisk$-surfaces) in $\mathbb{E}_1^4:=(\mathbb{R}^4,ds^2:=dx^2+dy^2+dz^2-dw^2)$.\\
    Along any spacelike, codimension 2 surface in $\mathbb{E}_1^4$, there exist exactly two independent, future-pointing (i.e., their timelike components will be positive), lightlike normal vector fields (upto scaling by a constant) $\bm{\nu}_1$ and $\bm{\nu}_2$ (section 2.1.1, \cite{senovilla2012trapped}). Then its mean curvature vector $\bm{H}=k_1\bm{\nu}_1+k_2\bm{\nu}_2$ where $k_1$ and $k_2$ are the components of the trace of the second fundamental form along $\bm{\nu}_1$ and $\bm{\nu}_2$ respectively.
    \begin{definition}
        A codimension 2 spacelike surface in $\mathbb{E}_1^4$ is called maximal if $k_1=0=k_2$, it is called weakly untrapped if $k_1k_2\leq 0$ and it is called a $\Asterisk$-surface if $k_1k_2\geq 0$ (see table 3 (Miscellaneous surfaces), \cite{senovilla2012trapped}).
    \end{definition}
For more details on the applications of weakly untrapped and $\Asterisk$-surfaces, see \cite{senovilla2012trapped,senovilla2007classification}.\\

    Consider the spacelike surface $F:\mathbb{R}^2\to\mathbb{E}_1^4$ defined by $(x,y)\mapsto (x,y,f(x,y),0)$. We can choose the lightlike normals $\bm{\nu}_1=(f_x,f_y,-1,\sqrt{1+f_x^2+f_y^2})$ and $\bm{\nu}_2=(-f_x,-f_y,1,\sqrt{1+f_x^2+f_y^2})$. Define $\mathbf{g}:=F^*ds^2$. Then it can be verified that;
    \begin{equation}\label{eq5.1}
        \begin{aligned}
            k_1 &=\mathrm{trace}_{\bm{\nu}_1}(I^{-1}II)=\mathbf{g}^{ij}F_{ij}.\bm{\nu}_1\\
            &=(1+f_y^2)f_{xx}+(1+f_x^2)f_{yy}-2f_xf_yf_{xy}=-k_2.
        \end{aligned}
    \end{equation}
    Similarly, the spacelike surface $G(x,y):=(0,y,x,g(x,y))$ into $\mathbb{E}_1^4$ in the domain (where it's spacelike) $\{(x,y)|\hspace{0.05cm}g_x^2+g_y^2<1\}$ has future pointing lightlike normal vector fields $\bm{\tilde{\nu}}_1:=(\sqrt{1-g_y^2-g_x^2},g_y,g_x,1)$ and $\bm{\tilde{\nu}}_2:=(-\sqrt{1-g_y^2-g_x^2},g_y,g_x,1)$. It can be verified for $\tilde{\mathbf{g}}:=G^*ds^2$ that;
    \begin{equation}\label{eq5.2}
        \begin{aligned}
            k_1 &=(1-g_y^2)g_{xx}+(1-g_x^2)g_{yy}+2g_xg_yg_{xy}=k_2.
        \end{aligned}
    \end{equation}
    Likewise, the spacelike surface $H(y,z):=(0,h(y,z),z,y)$ into $\mathbb{E}_1^4$ in the domain (where it's spacelike) $\{(y,z)|\hspace{0.05cm}h_y^2>1+h_z^2\}$ has future pointing lightlike normal vector fields $\bm{\overline{\nu}}_1:=(\sqrt{-1-h_z^2+h_y^2},1,-h_z,h_y)$ and $\bm{\overline{\nu}}_2:=(\sqrt{-1-h_z^2+h_y^2},-1,h_z,-h_y)$. It can be verified for $\overline{\mathbf{g}}:=H^*ds^2$ that;
    \begin{equation}\label{eq5.3}
        \begin{aligned}
            k_1 &=(1+h_z^2)h_{yy}-(1-h_y^2)h_{zz}-2h_yh_zh_{yz}=-k_2.
        \end{aligned}
    \end{equation}
    As a consequence of equations \eqref{eq5.1},\eqref{eq5.2} and \eqref{eq5.3}, we have the following facts:
    \begin{enumerate}
            \item The spacelike surfaces $F$ and $H$ are weakly untrapped and they are (respectively) maximal iff  $f$ and $h$ (respectively) satisfy the equations \eqref{mse} and \eqref{BIE}.
            \item The spacelike surface $G$ is a $\Asterisk$-surface and it is maximal iff $g$ satisfies equation \eqref{MaSE}.
    \end{enumerate}
    
    \begin{remark}\label{remark5.2}
        Now, keeping the consequences noted above in hindsight, let the various decompositions derived in the previous sections be reconsidered. It is easy to see that each of the height functions of the dilated zero mean curvature surfaces (in the RHS of the decomposition identities) which \enquote{sum up} to produce height functions of minimal and maximal surfaces (in the LHS of the decomposition identities) in $\mathbb{E}^3$ and $\mathbb{E}_1^3$, when (spacelike) immersed into $\mathbb{E}_1^4$ through maps of type $F,G$ and $H$ correspond to the functions $f,g$ and $h$ respectively. As a result of this, given any dilated zero mean curvature surface as in the individual summands of the RHS of the decompositions constructed in this paper, if it can be immersed into $\mathbb{E}_1^4$ through $F$ and $H$, its spacelike loci shall constitute weakly untrapped surfaces and if it can be immersed into $\mathbb{E}_1^4$ through $G$, its spacelike loci shall constitute $\Asterisk$-surfaces.\\
       Thus, the decompositions in \cite{kamien2001decomposition} can be seen as the Scherk's minimal surface in $\mathbb{E}^3$, now immersed into $\mathbb{E}_1^4$ as a (spacelike) maximal codimension 2 surface, expressed as an \enquote{infinite sum} of weakly untrapped surfaces (dilated helicoids). Similarly, equation \eqref{decompositionmaximal} expresses the Scherk's maximal surface in $\mathbb{E}_1^3$ as an \enquote{infinite sum} of $\Asterisk$-surfaces (dilated helicoids) in $\mathbb{E}_1^4$ while equation \eqref{decompositiontimelike} expresses the Scherk type Born-Infeld soliton in $\mathbb{E}_1^3$ as an \enquote{infinite sum} of weakly untrapped surfaces (dilated hyperbolic helicoids) in $\mathbb{E}_1^4$.\\
        Likewise, from Theorem \ref{theorem finite decomposition}, the Scherk type surfaces (on the LHS) in $\mathbb{E}^3$ and $\mathbb{E}_1^3$ when suitably immersed into $\mathbb{E}_1^4$ (through the maps $F, G$ and $H$), their spacelike loci can be expressed as \enquote{finite sums} of weakly untrapped and $\Asterisk$-surfaces (dilated Scherk type surfaces).
    \end{remark}
    For the previous Remark \ref{remark5.2}, it is imperative to note that the dilated helicoids and Scherk type surfaces appearing in all our derived decompositions shall (almost) always have non-vanishing mean curvature. Even when the mean curvature does vanish along these surfaces, it must only happen along isolated points or a nowhere dense loci since otherwise, if there did exist regions (open sets) on these surfaces where zero mean curvature condition was satisfied, those regions would necessarily experience no dilation at all, contradicting our setup.\\\\  
    On a side note we also have the following easy generalization for $F$;
    \begin{remark}
        The map $F(x,y):=(x,y,f(x,y),0)$ into the (flat) FLRW spacetime $(\mathbb{R}^4,(a(w))^2(dx^2+dy^2+dz^2)-dw^2)$ is a maximal surface (in $\mathbb{E}_1^4$) iff $f$ satisfies \eqref{mse}. 
    \end{remark}
    \section{Acknowledgement}
  The first and second authors of this paper acknowledge the funding received from UGC, India. The third author acknowledges funding by CSIR-JRF and SRF grants (File no. 09/1023(12774)/2021-EMR-I). The fourth author is partially supported by the MATRICS grant (File No. MTR/2023/000990), which has been sanctioned by the SERB.
\bibliography{ref}
\bibliographystyle{ieeetr}
\end{document}